\documentclass{amsart}
\usepackage{txfonts}
\usepackage{url}
\usepackage{amssymb}

\newtheorem{thm}{Theorem}
\newtheorem *{cor}{Corollary}
\newtheorem{lem}{Lemma}
\newtheorem *{prop}{Proposition}

\newtheorem *{hyp}{Hypothesis}

\newtheorem *{mainthm*}{Main Theorem}

\theoremstyle{definition}
\newtheorem{defin}{Definition}

\newtheorem*{xrem}{Remark}

\numberwithin{equation}{section}

\DeclareMathOperator*{\psum}{{\sum}^\prime}

\newcommand{\chisum}
{\sum_{\substack{\chi_1\ppmod{q_1}\\\chi_2\ppmod{q_2}}}}

\newcommand{\ppmod}[1]{\ (\mathrm{mod}\ {#1})}
\renewcommand{\phi}{\varphi}
\renewcommand{\epsilon}{\varepsilon}
\renewcommand{\Re}{\mathrm{Re}\,}
\renewcommand{\Im}{\mathrm{Im}\,}

\begin{document}
\title[the representation function for the sum of  two primes in A.P.]
{A mean value of  the representation function
\\for the sum of two primes in arithmetic progressions}

\author[Y. Suzuki]{Yuta Suzuki}
\address{Graduate School of Mathematics\\
Nagoya University\\
Chikusa-ku, Nagoya 464-8602, Japan.}
\email{m14021y@math.nagoya-u.ac.jp}
\date{}

\begin{abstract}
In this note, assuming a variant of the Generalized Riemann Hypothesis,
which does not exclude the existence of real zeros,
we prove an asymptotic formula 
for the mean value of the representation function
for the sum of two primes in arithmetic progressions.
This is an improvement of the result of F. R\"uppel in 2009,
and the generalization of the result of A. Languasco and A. Zaccagnini
concerning the ordinary Goldbach problem in 2012.
\end{abstract}


\maketitle

\section{Introduction}
In this note, we consider the sum of two primes in arithmetic progressions.
For the conventional studies on this additive problem,
for example,
see Lavrik \cite{Lavrik} or Liu and Zhan \cite{LIU_ZHAN}.
They gave some estimates on the exceptional set for this problem.
In the following, although it is rather an indirect way,
we shall consider this problem in some average sense as R\"uppel did in
\cite{Ruppel1,Ruppel2}.

For this additive problem, the usual weighted representation function is given by
\[R(n,q_1,a_1,q_2,a_2):=
\sum_{\substack{m_1+m_2=n\\m_i\equiv a_i\pmod{q_i}}}
\Lambda(m_1)\Lambda(m_2),\]
where $\Lambda(n)$ is the von Mangoldt function
and $a_1,a_2,q_1,q_2,n$ are positive integers satisfying $(a_1,q_1)=(a_2,q_2)=1$.
Let us also introduce an abbreviation
\[R(n,q,a,b):=R(n,q,a,q,b)\]
for positive integers $a,b,q$ satisfying $(ab,q)=1$.
In 2009, R\"uppel \cite{Ruppel1} studied the mean value
of this representation function, i.e.
\begin{equation}
\label{one_modulus}
\sum_{n\le X}R(n,q,a,b).
\end{equation}
In particular, she obtained,
under a weakened variant of the Generalized Riemann Hypothesis,
an asymptotic formula for the mean value (\ref{one_modulus}).
More precisely, she assumed the Generalized Riemann Hypothesis (GRH)
except the existence of real zeros,
i.e. she assumed
\footnote{Her hypothesis was actually somewhat stronger than ours.}
\begin{hyp}[GRH with real zeros]
\label{GRHR}
Every complex non-trivial
\footnote{
Here we say that a zero $\rho$ of a Dirichlet $L$ function is \textit{non-trivial}
if $0<\Re \rho <1$, and \textit{complex} if $\Im\rho\neq0$.
}
zeros of Dirichlet L functions lie on the critical line $\sigma=1/2$.
\end{hyp}
In this note, we also assume this hypothesis following her and
we call this hypothesis GRHR shortly.
Note that the Riemann zeta function has no real non-trivial zero.

Assuming this hypothesis GRHR, R\"uppel \cite{Ruppel1} proved
\footnote{
This is somewhat different from what R\"uppel claimed,
but she essentially proved this.
}
\begin{equation}
\label{ruppel}
\sum_{n\le X}R(n,q,a,b)
=\frac{X^2}{2\phi(q)^2}+O\left(X^{1+\delta}(\log q)^2\right),
\end{equation}
where $\delta=1/2$ unless real zeros exist for the modulus $q$,
in which case we let $\delta$ be the largest one among these real zeros.

She considered the mean value (\ref{one_modulus}),
but we can also obtain the corresponding result for the mean value
\begin{equation}
\label{two_modulus}
\sum_{n\le X}R(n,q_1,a_1,q_2,a_2)
\end{equation}
through the same method.
Moreover, her method can be used to prove the asymptotic formula
\begin{equation}
\begin{aligned}
\label{true_Ruppel}
\sum_{n\le X}R(n,q,a,b)&=\frac{X^2}{2\phi(q)^2}
-\sum_{\chi\ppmod{q}}\frac{\overline{\chi}(a)+\overline{\chi}(b)}{\phi(q)^2}
\sum_{\beta_\chi}\frac{X^{\beta_\chi+1}}{\beta_\chi(\beta_\chi+1)}\\
&+\sum_{\chi,\,\psi\ppmod{q}}
\frac{\overline{\chi}(a)\overline{\psi}(b)}{\phi(q)^2}
\sum_{\beta_\chi,\,\beta_\psi}
\frac{\Gamma(\beta_\chi)\Gamma(\beta_\psi)}{\Gamma(\beta_\chi+\beta_\psi)}
\frac{X^{\beta_\chi+\beta_\psi}}{\beta_\chi+\beta_\psi}
+O\left(X^{3/2}(\log q)^2\right),
\end{aligned}
\end{equation}
where $\beta_\chi$ runs through all real zeros of $L(s,\chi)$
with $\beta_\chi\ge 1/2$.

These results correspond to the result of Fujii \cite{Fujii1}
for the ordinary Goldbach Problem.
Fujii \cite{Fujii1} proved, for the representation function
\[R(n)=\sum_{k+l=n}\Lambda(k)\Lambda(l)\]
of the ordinary Goldbach problem, an asymptotic formula
\[\sum_{n\le X}R(n)=\frac{X^2}{2}+O\left(X^{3/2}\right)\]
under the Riemann Hypothesis (RH).
This was improved by Fujii \cite{Fujii2} himself to
\begin{equation}
\label{Fujii}
\sum_{n\le X}R(n)=\frac{X^2}{2}
-2\sum_\rho\frac{X^{\rho+1}}{\rho(\rho+1)}+O\left((X\log X)^{4/3}\right)
\end{equation}
under RH, where $\rho$ runs through all non-trivial zeros of Riemann zeta function.
After this pioneering work of Fujii,
the error term on the right-hand side of (\ref{Fujii}) was improved to
\begin{align*}
\ll X(\log X)^5\quad&
(\text{by Bhowmik and Schlage-Puchuta \cite{BP1}}),\\
\ll X(\log X)^3\quad&
(\text{by Languasco and Perelli \cite{LZG}}).
\end{align*}
Of course we assume RH in all of these results.
We note that Bhowmik and Schlage-Puchuta proved also the omega result
\begin{equation}
\label{Omega_BS}
=\Omega(X\log\log X)
\end{equation}
for this error term. This omega result is independent from RH or GRH.

In this note, we improve R\"uppel's result (\ref{ruppel}) up to the accuracy
of the result of Languasco and Zaccagnini \cite{LZG}.
Let us introduce
\[W(X,z,w):=\frac{\Gamma(z)\Gamma(w)}{\Gamma(z+w)}
\cdot\frac{X^{z+w}}{z+w},\]
\[G^\kappa(X,\chi):=\sum_{\rho_\chi}W(X,\rho_\chi,\kappa),\quad
G(X,\chi)=G^1(X,\chi),\]
\[G^\kappa(X,q,a):=\sum_{\chi\ppmod{q}}\overline{\chi}(a)G^\kappa(X,\chi), \quad
G(X,q,a)=G^1(X,q,a),\]
where $\chi\ppmod{q}$ are Dirichlet characters,
$\rho_\chi$ runs through all non-trivial zeros of Dirichlet $L$ function $L(s,\chi)$
including real zeros, and $\kappa>0$.
Then our result is:
\begin{thm}
\label{main_thm}
Assume GRHR.
For $X\ge 2$ and positive integers $a_1,a_2,q_1,q_2$
with $(a_1,q_1)=1$ and $(a_2,q_2)=1$, we have
\begin{align*}
\sum_{n\le X}R(n,q_1,a_1,q_2,a_2)
=\frac{1}{\phi(q_1)\phi(q_2)}\left(\frac{X^2}{2}
-G(X,q_1,a_1)-G(X,q_2,a_2)+H(X)\right)+E(X)
\end{align*}
with%
\footnote{Rigorously speaking,
these functions should be written as
$H(X,q_1,a_1,q_2,a_2)$ and $E(X,q_1,a_1,q_2,a_2)$,
but in the following, we use such abbreviations
if there is no possibility of confusion.}
\begin{multline*}
H(X)=\chisum\overline{\chi_1}(a_1)\overline{\chi_2}(a_2)
\left(\sum_{\beta_2}G^{\beta_2}(X,\chi_1)
+\sum_{\beta_1}G^{\beta_1}(X,\chi_2)
-\sum_{\beta_1,\,\beta_2}W(X,\beta_1,\beta_2)\right),
\end{multline*}
\[E(X)\ll X(\log X)(\log q_1X)(\log q_2X),\]
where $\beta_i$ runs through all real non-trivial zeros\footnote{
Therefore, if there are no real zeros,
then the sum over $\beta_i$ turns to be empty.} of $L(s,\chi_i)$
with $\beta_i\ge 1/2$ for $i=1,2$,
and the implicit constant in the error term is absolute.
\end{thm}

If we compare our result with R\"uppel's result,
then we find that in her asymptotic formula (\ref{ruppel}),
the terms correspond to the oscillating terms
\[G(X,q_1,a_1),\ G(X,q_2,a_2),\ H(X)\]
of Theorem \ref{main_thm} are included
in the error term $O(X^{1+\delta}(\log q)^2)$.

R\"uppel obtained the meromorphic continuation of the function
\[\Phi_2(s,a,b,q)
=\sum_{n=1}^\infty \frac{R(n,a,q,b,q)}{n^s}
=\sum_{\substack{k,m=1\\k\equiv a\pmod{q}\\m\equiv b\pmod{q}}}^\infty
\frac{\Lambda(k)\Lambda(m)}{(k+m)^s}\]
to the half plane $\Re s>1$.
This analytic approach to the problem which we are considering is initiated by
Egami and Matsumoto \cite{EM} in 2007.
Meromorphic continuation of the above type relates
to the mean value of the same type as (\ref{two_modulus})
through Perron's formula.
Such a correspondence also exists%
\footnote{See the Remark at the last of this note.}
between R\"uppel's continuation and Theorem \ref{main_thm}.

Our proof of  Theorem \ref{main_thm} follows
the argument of Languasco and Zaccagnini \cite{LZG}
except that we need to treat the real zeros which we permit.
And in order to treat these zeros,
we have to modify the estimate of
Languasco and Perelli \cite{LP}, i.e. we have to treat the terms
arising from real zeros like the main term.
\section{Notations}
Here we briefly summarize the notations which we use in this note.
Some exceptional notations are given at each occurrence.
\begin{eqnarray*}
&x,\alpha&: \text{real numbers,}\\
&X,T\ge 2&: \text{real numbers,}\\
&N,M\ge 2&: \text{positive integers,}\\
&a_1,q_1,a_2,q_2&: \text{positive integers
satisfying $(a_1,q_1)=1$ and $(a_2,q_2)=1$,}\\
&a,b,q&: \text{positive integers satisfying $(ab,q)=1$,}\\
&m,n&: \text{integers
(We impose $m,n\ge1$ when these are used as summation variables.),}\\
&p&: \text{prime numbers,}\\
&\phi(q)&: \text{Euler totient function,}\\
&\Lambda(n)&: \text{von Mangoldt function,}\\
&e(x)&:=\exp(2\pi i x),\\
&\chi\ppmod{q}&: \text{Dirichlet characters of modulus $q$,}
\end{eqnarray*}

Chebyshev functions:
\[\psi(x):=\sum_{n\le x}\Lambda(n),\quad
\psi(x,q,a):=\sum_{\substack{n\le x\\n\equiv a\ppmod{q}}}\Lambda(n),\quad
\psi(x,\chi):=\sum_{n\le x}\chi(n)\Lambda(n).\]

For Dirichlet characters $\chi\ppmod{q}$, we define
\[E(\chi)=
\begin{cases}
0&(\text{when $\chi$ is non-principal}),\\
1&(\text{when $\chi$ is principal}).
\end{cases}\]
We regard the arithmetic function whose value is always constant 1
as the primitive Dirichlet character of modulus 1,
and the other principal Dirichlet characters of modulus $q\ge 2$
as imprimitive characters.
For an imprimitive character $\chi\ppmod{q}$,
we denote by $\chi^\ast\ppmod{q^\ast}$
the primitive character which induces $\chi\ppmod{q}$.
We use a complex variable $s=\sigma+it$
where $\sigma$ is the real part of $s$ and $t$ is the imaginary part of $s$.

We say that a zero $\rho$ of Dirichlet $L$ functions is \textit{non-trivial}
if $\rho$ satisfies $0<\Re\rho<1$,
and we will denote by $\rho$ with imaginary part $\gamma$
the non-trivial zeros of $L(s,\chi)$.
If we should show explicitly to which character these zeros correspond,
then we attach the character as its suffix, e.g. $\rho_\chi$.
Moreover, $\beta$ denotes real zeros of Dirichlet $L$ functions, and
if $\beta$ is used as a summation variable, then it runs through
all of the real zeros of a given Dirichlet $L$ function with $\ge 1/2$.
We rigorously distinguish two terms \textit{real zero} and \textit{Siegel zero},
see Section \ref{section_Siegel}.
\section{Classical results and remarks on the Siegel zeros}
\label{section_Siegel}
In this section, we recall some classical results on Siegel zeros.
We also introduce notations and conventions for real and Siegel zeros.

It is well-known that there exists an absolute constant
\[c_1>0\]
satisfying the following theorem of Landau.
We use this letter $c_1$ always as the same meaning throughout this paper.
First we introduce the following definition.
\begin{defin}
Let $q\ge 2$ and $\chi\ppmod{q}$ be a Dirichlet character.
A real non-trivial zero $\beta$ of $L(s,\chi)$ is called a \textit{Siegel zero} if
\[\beta>1-\frac{c_1}{\log q}.\]
If $L(s,\chi)$ has a Siegel zero, then we say $\chi\ppmod{q}$ has a Siegel zero.
\end{defin}
Then the well-known theorem of Landau
\footnote{Of course, we weaken the original Landau's theorem
up to the form we need.}
is the following \cite[Corollary 11.8]{MV}.
\begin{thm}
\label{Siegel}
Let $q$ be a positive integer. Then among Dirichlet characters $\chi\ppmod{q}$
of the modulus $q$, there is at most one character which has a Siegel zero.
Moreover, if such a character $\psi\ppmod{q}$ with Siegel zero exists,
then $L(s,\psi)$ has only one Siegel zero even counting with multiplicity.
\end{thm}

Hence in this note we do not call
all real non-trivial zeros of Dirichlet $L$ function Siegel zero,
although R\"uppel used such a terminology.
We denote by $\beta_0$ the Siegel zeros.
If we should show explicitly $\beta_0$ is
the one corresponds to the character $\chi\ppmod{q}$,
then we write like $\beta_0(\chi)$.
In the following of this paper, there are some terms containing Siegel zeros.
We ignore such terms if corresponding Siegel zeros do not exist.

We also recall the following approximation \cite[Theorem 11.4]{MV}
of $L'/L(1,\chi)$.
\begin{lem}
\label{log_dev}
For a non-principal character $\chi\ppmod{q}$, we have
\[\frac{L'}{L}(1,\chi)=\frac{1}{1-\beta_0}+O(\log q),\]
where $\beta_0$ is the Siegel zero of $L(s,\chi)$.
\end{lem}
\section{A modification of the estimate of Languasco and Perelli with real zeros}
We consider the following generating functions
\[\widetilde{S}(\alpha)=\sum_{n=1}^\infty\Lambda(n)e^{-n/N}e(n\alpha),\]
\[\widetilde{S}(\alpha,\chi)
=\sum_{n=1}^\infty\chi(n)\Lambda(n)e^{-n/N}e(n\alpha).\]
If we introduce the argument $z=1/N-2\pi i\alpha$, then we can express them as
\[\widetilde{S}(\alpha)=\sum_{n=1}^\infty\Lambda(n)e^{-nz},\quad
\widetilde{S}(\alpha,\chi)=\sum_{n=1}^\infty\chi(n)\Lambda(n)e^{-nz}.\]
Let us first recall the following explicit formula for these series.

\begin{lem}
For $N\ge2$, $\alpha\in[-1/2,1/2]$ and a primitive character $\chi\pmod{q}$,
we have
\begin{equation}
\label{explicit1}
\widetilde{S}(\alpha)
=\frac{1}{z}-\sum_\rho z^{-\rho}\,\Gamma(\rho)+O(1),\quad q=1,
\end{equation}
\begin{equation}
\label{explicit2}
\widetilde{S}(\alpha,\chi)
=-\sum_\rho z^{-\rho}\,\Gamma(\rho)
+\frac{L'}{L}(1,\overline{\chi})+O(\log qN),\quad q\ge2,
\end{equation}
where $\rho$'s are the non-trivial zeros of $L(s,\chi)$ including real zeros.
\end{lem}
This is an easy application of the Mellin-Cahen formula
and we can obtain this explicit formula unconditionally.
For the proof, see \cite[Lemma 2]{LZ_many}

The modified version of the estimate of Languasco and Perelli is the following.
\begin{thm}
\label{LP_real}
Assume GRHR.
For $N\ge 2,\ 0<\xi\le1/2$, and a Dirichlet character $\chi \pmod{q}$, we have
\[\int_{-\xi}^\xi\left|\widetilde{S}(\alpha,\chi)-\frac{E(\chi)}{z}
+\sum_\beta\frac{\Gamma(\beta)}{z^\beta}\right|^2d\alpha
\ll N\xi(\log qN)^2\]
where $\beta$ runs through
the all real non-trivial zeros of $L(s,\chi)$ satisfying $\ge 1/2$.
\end{thm}
For the proof of this theorem,
we need to modify only the preliminary calculations
in the original proof of Languasco and Perelli.
The remaining and most important part of the proof is exactly the same one
which Languasco and Perelli give.
Hence we see only how the preliminary calculations should be modified.
\begin{proof}
At first we assume that $\chi$ is primitive.
In the case where $\chi$ is principal, i.e. the case $q=1$,
there is no real zeros.
So the theorem reduces to the original estimate of Languasco and Perelli.
Next we consider the case where $\chi$ is non-principal or $q\ge 2$.
Then the integral we are considering becomes
\begin{equation}
\label{integral_non_principal}
\int_{-\xi}^\xi\left|\widetilde{S}(\alpha,\chi)
+\sum_\beta\frac{\Gamma(\beta)}{z^\beta}\right|^2d\alpha.
\end{equation}
We shall substitute the explicit formula (\ref{explicit2}) into this integral.
Before this substitution, we reform (\ref{explicit2}) slightly.
By the usual approximation for $L(s,\chi)$, we get
\[\frac{L'}{L}(1,\chi)=\sum_{|\gamma|<1}\frac{1}{1-\rho}+O(\log q).\]
By GRHR, we can estimate this as
\[\frac{L'}{L}(1,\chi)=\sum_\beta\frac{1}{1-\beta}+O(\log q),\]
because there is at most $O(\log q)$ zeros in the region $|\gamma|<1$.
Let us substitute this expression into the explicit formula (\ref{explicit2}).
We obtain
\begin{align*}
\widetilde{S}(\alpha,\chi)
=&-\sum_\rho z^{-\rho}\,\Gamma(\rho)
+\sum_\beta\frac{1}{1-\beta}+O(\log qN)\\
=&-\psum_\rho z^{-\rho}\,\Gamma(\rho)
-\sum_\beta\frac{\Gamma(\beta)}{z^{\beta}}
-\sum_\beta\left(\frac{\Gamma(1-\beta)}{z^{1-\beta}}-\frac{1}{1-\beta}\right)
+O\left(N^{1/2}(\log qN)\right),
\end{align*}
where $\psum$ is the sum excluding real zeros.
Now by the existence of the pole 0 of the gamma function,
if $3/4\le\beta<1$, we have
\begin{align*}
\frac{\Gamma(1-\beta)}{z^{1-\beta}}-\frac{1}{1-\beta}
=&\frac{1}{1-\beta}\left(z^{\beta-1}-1\right)+O\left(\frac{1}{|z|^{1-\beta}}\right)\\
=&-\frac{\log z}{1-\beta}\int_0^{1-\beta}z^{-\sigma}d\sigma
+O\left(\frac{1}{|z|^{1-\beta}}\right)
\ll N^{1/4}(\log N)\ll N^{1/2},
\end{align*}
where we used the fact $1/N\ll|z|\ll1$.
If $1/2\le\beta\le3/4$, then we can use the following trivial estimate
\[\frac{\Gamma(1-\beta)}{z^{1-\beta}}-\frac{1}{1-\beta}\ll N^{1/2}.\]
Because there exists at most $O(\log q)$ real zeros, we have
\[\sum_\beta\left(\frac{\Gamma(1-\beta)}{z^{1-\beta}}-\frac{1}{1-\beta}\right)
\ll N^{1/2}(\log q).\]
In this estimate, we have to recall that
we imposed the restriction $\beta\ge1/2$ on the variable $\beta$.
Hence our reformed explicit formula takes the form
\begin{equation}
\label{explicit_reformed}
\widetilde{S}(\alpha,\chi)=
-\psum_\rho z^{-\rho}\,\Gamma(\rho)
-\sum_\beta\frac{\Gamma(\beta)}{z^{\beta}}
+O\left(N^{1/2}(\log qN)\right).
\end{equation}
Hence the integrand of (\ref{integral_non_principal}) can be rewritten as
\[\left|\widetilde{S}(\alpha,\chi)
+\sum_\beta\frac{\Gamma(\beta)}{z^\beta}\right|^2
\ll\left|\psum_\rho z^{-\rho}\,\Gamma(\rho)\right|^2
+N(\log qN)^2.\]
Therefore our integral (\ref{integral_non_principal}) is reduced to
\begin{equation}
\label{pre_R}
\int_{-\xi}^\xi\left|\widetilde{S}(\alpha,\chi)
+\sum_\beta\frac{\Gamma(\beta)}{z^\beta}\right|^2d\alpha
\ll R+N\xi(\log qN)^2
\end{equation}
where $R$ is given by
\[R=\int_{-\xi}^\xi\left|\psum_\rho z^{-\rho}\,\Gamma(\rho)\right|^2d\alpha.\]
In order to estimate this $R$,
we can repeat the argument of Languasco and Perelli straightforwardly.
For example,
see Section 2 of \cite{LP} or the proof of Lemma 7 in \cite{LZ_many}.
Their argument gives the estimate
\[R\ll N\xi(\log qN)^2.\]
Substituting this estimate into (\ref{pre_R}), we finally get
\[\int_{-\xi}^\xi\left|\widetilde{S}(\alpha,\chi)
+\sum_\beta\frac{\Gamma(\beta)}{z^\beta}\right|^2d\alpha
\ll N\xi(\log qN)^2\]
for all primitive characters $\chi\ppmod{q}$.

Now we prove the theorem for all characters $\chi\ppmod{q}$,
i.e. we now do not assume that $\chi\ppmod{q}$ is primitive.
Let $q^\ast$ be the conductor of $\chi\ppmod{q}$ and $\chi^\ast\ppmod{q^\ast}$
be the primitive character which induces $\chi\ppmod{q}$.
Then by the above argument and the original result of Languasco and Perelli,
we have
\[\int_{-\xi}^\xi\left|\widetilde{S}(\alpha,\chi^\ast)
-\frac{E(\chi)}{z}+\frac{\Gamma(\beta)}{z^\beta}\right|^2d\alpha
\ll N\xi(\log q^\ast N)^2\ll N\xi(\log qN)^2.\]
Moreover, we have
\[\widetilde{S}(\alpha,\chi)-\widetilde{S}(\alpha,\chi^\ast)
\ll\sum_{p|q}(\log p)\sum_{k=1}^\infty e^{-p^k/N}\ll(\log q)(\log N)
\ll N^{1/2}(\log qN).\]
Therefore, we obtain the theorem for all Dirichlet characters.
\end{proof}
\section{Detection and calculations of the main and oscillating terms}
Let us now consider the integrals
\begin{equation}
\label{main_integral}
\int_{-1/2}^{1/2}\frac{T(y,-\alpha)}{z^\mu}d\alpha,
\quad\int_{-1/2}^{1/2}T(y,-\alpha)\frac{\widetilde{S}(\alpha,\chi)}{z^\mu}d\alpha,
\end{equation}
where $\mu$ is some positive constant
and $T(y,\alpha)$ is given by
\[T(y,\alpha)=\sum_{m\le y}e(m\alpha).\]
We begin with the following integral formula.
\begin{lem}
\label{integral_formula}
For integers $N,n$ with $N\ge 2$ and a positive real number $\mu>0$,
we have
\[\int_{-1/2}^{1/2}\frac{e(-n\alpha)}{z^\mu}d\alpha=
\begin{cases}
\displaystyle e^{-n/N}\cdot\frac{n^{\mu-1}}{\Gamma(\mu)}
+O\left(\frac{2^\mu}{|n|}\right)&(\text{when $n>0$}),\\
\displaystyle O\left(\frac{2^\mu}{|n|}\right)&(\text{when $n<0$}),\\
\displaystyle O(\log N)&(\text{when $n=0,\ 0<\mu\le1$}).
\end{cases}\]
\end{lem}

\begin{proof}
We first consider the case $n\neq0$.
For arbitrary $Y>X\ge0$, by integration by parts, we have
\begin{align*}
\int_X^Y\frac{e(-n\alpha)}{z^\mu}d\alpha
=&\int_X^Y\frac{e(-n\alpha)}{(1/N-2\pi i\alpha)^\mu}d\alpha\\
=&-\frac{1}{2\pi in}
\left[\frac{e(-n\alpha)}{(1/N-2\pi i\alpha)^\mu}\right]_X^Y
+\frac{\mu}{n}\int_X^Y\frac{e(-n\alpha)}{(1/N-2\pi i\alpha)^{\mu+1}}d\alpha\\
\ll&\frac{1}{|n|X^\mu}+\frac{\mu}{|n|}\int_X^Y\frac{d\alpha}{\alpha^{\mu+1}}
\ll\frac{1}{|n|X^\mu}.
\end{align*}
Therefore we can extend the integral with the following error term:
\[\int_{-1/2}^{1/2}\frac{e(-n\alpha)}{z^\mu}d\alpha
=\int_{-\infty}^{\infty}\frac{e(-n\alpha)}{z^\mu}d\alpha
+O\left(\frac{2^\mu}{|n|}\right).\]
Now we rewrite this completed integral into the complex integral as
\[\int_{-\infty}^{\infty}\frac{e(-n\alpha)}{z^\mu}d\alpha
=\frac{e^{-n/N}}{2\pi i}
\int_{1/N-i\infty}^{1/N+i\infty}\frac{\exp(ns)}{s^\mu}ds.\]
Besides $n\neq0$, let us also assume $n>0$. Let $T,K>0$ be positive real numbers,
and consider the contours of integration
\begin{itemize}
\item $C_1$: The line segment from $1/N-iT$ to $-K-iT$,
\item $C_2$: The line segment from $-K-iT$ to $-K-i/n$,
\item $C_3$: The line segment from $-K+i/n$ to $-K+iT$,
\item $C_4$: The line segment from $-K+iT$ to $1/N+iT$,
\item $H_1(K)$: The line segment from $-K-i/n$ to $-i/n$,
\item $H_2(K)$: The half circle from $-i/n$ to $i/n$\\
\quad\quad\quad\quad with center 0 and radius $1/n$  in the positive direction,
\item $H_3(K)$: The line segment from $i/n$ to $-K+i/n$,
\end{itemize}
and
\[\mathcal{H}(K)=H_1(K)+H_2(K)+H_3(K),\ \mathcal{H}=\mathcal{H}(\infty).\]
We shift the contour to the left to obtain
\[\frac{e^{-n/N}}{2\pi i}
\int_{1/N-i\infty}^{1/N+i\infty}\frac{\exp(ns)}{s^\mu}ds
=\frac{e^{-n/N}}{2\pi i}
\int_{\mathcal{H}(K)}\frac{\exp(ns)}{s^\mu}ds
+\frac{e^{-n/N}}{2\pi i}
\int_{C_1+C_2+C_3+C_4}\frac{\exp(ns)}{s^\mu}ds.\]
Here we have
\[\int_{C_1+C_4}\frac{\exp(ns)}{s^\mu}ds\ll_{n,K}\frac{1}{T^\mu}.\]
Hence letting $T\to\infty$, we get
\begin{align*}
\frac{e^{-n/N}}{2\pi i}
\int_{1/N-i\infty}^{1/N+i\infty}\frac{\exp(ns)}{s^\mu}ds
=\frac{e^{-n/N}}{2\pi i}
&\int_{\mathcal{H}(K)}\frac{\exp(ns)}{s^\mu}ds\\
&+\frac{e^{-n/N}}{2\pi i}
\left(\int_{-K-i\infty}^{-K-i/n}+\int_{-K+i/n}^{-K+i\infty}\right)
\frac{\exp(ns)}{s^\mu}ds.
\end{align*}
For one of the last two integrals, we get
\begin{align*}
&\int_{-K-i\infty}^{-K-i/n}\frac{\exp(ns)}{s^\mu}ds
=2\pi ie^{-nK}\int_{1/2\pi n}^{\infty}
\frac{e(-n\alpha)}{(-K-2\pi i\alpha)^\mu}d\alpha\\
=&-\frac{e^{-nK}}{n}
\left[\frac{e(-n\alpha)}{(-K-2\pi i\alpha)^\mu}\right]_{1/2\pi n}^{\infty}
+\frac{2\pi i\mu e^{-nK}}{n}\int_{1/2\pi n}^{\infty}
\frac{e(-n\alpha)}{(-K-2\pi i\alpha)^{\mu+1}}d\alpha\\
\ll&\frac{1}{K^\mu}
+\frac{\mu}{|n|}\int_0^\infty\frac{1}{(K^2+\alpha^2)^{(\mu+1)/2}}d\alpha
\ll_\mu \frac{1}{K^\mu}.
\end{align*}
The other can be estimated similarly.
Hence letting $K\to\infty$, we get
\[\frac{e^{-n/N}}{2\pi i}
\int_{1/N-i\infty}^{1/N+i\infty}\frac{\exp(ns)}{s^\mu}ds
=\frac{e^{-n/N}}{2\pi i}
\int_{\mathcal{H}}\frac{\exp(ns)}{s^\mu}ds\]
and by the Hankel integral formula for the gamma function, we have
\[\frac{e^{-n/N}}{2\pi i}
\int_{1/N-i\infty}^{1/N+i\infty}\frac{\exp(ns)}{s^\mu}ds
=n^{\mu-1}\frac{e^{-n/N}}{2\pi i}
\int_{n\mathcal{H}}\frac{\exp(s)}{s^\mu}ds
=e^{-n/N}\frac{n^{\mu-1}}{\Gamma(\mu)}.\]
Therefore we have proved
\[\int_{-1/2}^{1/2}\frac{e(-n\alpha)}{z^\mu}d\alpha=
e^{-n/N}\frac{n^{\mu-1}}{\Gamma(\mu)}+O\left(\frac{2^\mu}{n}\right)\]
for the case $n>0$. For the case $n<0$, we can prove the lemma by shifting the
contour of integration to the right and arguing as above.

Now we consider the remaining case $n=0$ and $0<\mu\le1$.
In this case, we can simply estimate
\[\int_{-1/2}^{1/2}\frac{e(-n\alpha)}{z^\mu}d\alpha
\ll N\int_0^{1/N}d\alpha+\int_{1/N}^{1/2}\frac{d\alpha}{\alpha}
\ll \log N\]
as desired, and we finally get the lemma for all cases.
\end{proof}

The first integral of (\ref{main_integral}) can be calculated as follows.
\begin{lem}
\label{t_detect}
For any positive integer $N\ge 2$ and real numbers $2<y,\ 0<\mu\le 2$,
we have
\[\int_{-1/2}^{1/2}\frac{T(y,-\alpha)}{z^\mu}d\alpha
=\frac{1}{\Gamma(\mu)}\sum_{m\le y}e^{-m/N}m^{\mu-1}
+O(\log y).\]
\end{lem}

\begin{proof}
Interchanging the order of summation and integration,
by Lemma \ref{integral_formula}, we have
\begin{align*}
\int_{-1/2}^{1/2}\frac{T(y,-\alpha)}{z^\mu}d\alpha
=&\sum_{m\le y}\int_{-1/2}^{1/2}\frac{e(-m\alpha)}{z^\mu}d\alpha\\
=&\frac{1}{\Gamma(\mu)}\sum_{m\le y}e^{-m/N}m^{\mu-1}
+O\left(\sum_{m\le y}\frac{1}{m}\right)\\
=&\frac{1}{\Gamma(\mu)}\sum_{m\le y}e^{-m/N}m^{\mu-1}
+O(\log y).\qedhere
\end{align*}
\end{proof}

In order to study the second integral of (\ref{main_integral}), we introduce
\[\psi_\mu(x,\chi):=\sum_{n<x}\chi(n)\Lambda(n)(x-n)^{\mu-1}\]
for Dirichlet character $\chi\ppmod{q}$. Then the result is the following.

\begin{lem}
\label{detect}
For any Dirichlet character $\chi\ppmod{q}$,
any positive integer $N\ge 2$, any real numbers $2<y$ and $\ 0<\mu\le 1$, we have
\[\int_{-1/2}^{1/2}T(y,-\alpha)\frac{\widetilde{S}(\alpha,\chi)}{z^\mu}d\alpha
=\frac{1}{\Gamma(\mu)}\sum_{m\le y}e^{-m/N}\psi_\mu(m,\chi)
+O\left(N(\log yN)+y(\log N)\right).\]
\end{lem}

\begin{proof}
First we interchange the order of summation and integration
\[\int_{-1/2}^{1/2}T(y,-\alpha)\frac{\widetilde{S}(\alpha,\chi)}{z^\mu}d\alpha
=\sum_{m\le y}\sum_{n=1}^\infty\chi(n)\Lambda(n)e^{-n/N}
\int_{-1/2}^{1/2}\frac{e(-(m-n)\alpha)}{z^\mu}d\alpha.\]
We apply Lemma \ref{integral_formula} to these last integrals,
and we have
\begin{align*}
=\frac{1}{\Gamma(\mu)}\sum_{m\le y}e^{-m/N}&
\sum_{n=1}^{m-1}\chi(n)\Lambda(n)(m-n)^{\mu-1}\\
+&O\left(2^\mu\sum_{m\le y}\sum_{\substack{n=1\\n\neq m}}^\infty
\Lambda(n)\frac{e^{-n/N}}{|m-n|}
+(\log N)\sum_{m\le y}\Lambda(m)e^{-m/N}\right).
\end{align*}
Here the error terms can be estimated as
\begin{align*}
2^\mu\sum_{m\le y}\sum_{n=m+1}^\infty
\Lambda(n)\frac{e^{-n/N}}{|m-n|}
\ll&\sum_{n=1}^\infty\Lambda(n)e^{-n/N}\sum_{m<n}\frac{1}{|m-n|}\\
\ll&\sum_{n=1}^\infty\Lambda(n)(\log2n)e^{-n/N}
\ll N(\log N),\\
2^\mu\sum_{m\le y}\sum_{n=1}^{m-1}
\Lambda(n)\frac{e^{-n/N}}{|m-n|}
\ll&\sum_{n=1}^\infty\Lambda(n)e^{-n/N}\sum_{n<m\le y}\frac{1}{|m-n|}\\
\ll&(\log y)\sum_{n=1}^\infty\Lambda(n)e^{-n/N}\\
\ll&N(\log y),
\intertext{and}
(\log N)\sum_{m\le y}\Lambda(m)e^{-m/N}\ll&\psi(y)(\log N)\ll y(\log N).
\end{align*}
Therefore, we have
\[\int_{-1/2}^{1/2}T(y,-\alpha)\frac{\widetilde{S}(\alpha,\chi)}{z^\mu}d\alpha
=\frac{1}{\Gamma(\mu)}\sum_{m\le y}e^{-m/N}\psi_\mu(m,\chi)
+O\left(N(\log yN)+y(\log N)\right).\qedhere\]
\end{proof}

We now calculate the main term of Lemma \ref{detect} explicitly.
\begin{lem}
\label{cal_osc}
Let $\chi\ppmod{q}$ be a Dirichlet character,
$M\ge2$ be a positive integer and $1/2<\mu\le1$ be a real number.
Moreover let $\chi^\ast\ppmod{q^\ast}$ be the primitive character
which induces $\chi\ppmod{q}$.
Then we have
\begin{align*}
\sum_{m=1}^M\psi_\mu(m,\chi)
=E(\chi)\frac{M^{\mu+1}}{\mu(\mu+1)}
-G^\mu(M,\chi)
+\frac{M^\mu}{\mu}\cdot\frac{L'}{L}(1,\overline{\chi^\ast})+O(M(\log2q)(\log M)),
\end{align*}
where the term
\[\frac{M^\mu}{\mu}\cdot\frac{L'}{L}(1,\overline{\chi^\ast})\]
appears only if $\chi\ppmod{q}$ is non-principal.
\end{lem}

\begin{proof}
First recalling $1/2<\mu\le1$ and using $\rho(x)=\{x\}-1/2$, we have
\begin{align*}
\sum_{n\le x}n^{\mu-1}
=&\int_{1-}^xu^{\mu-1}d[u]
=\int_{1}^xu^{\mu-1}du-\int_{1-}^xu^{\mu-1}d\rho(u)\\
=&\int_{0}^xu^{\mu-1}du+(\mu-1)\int_1^xu^{\mu-2}\rho(u)du+O(1)
=\int_{0}^xu^{\mu-1}du+O(1).
\end{align*}
Hence we get
\begin{align*}
\sum_{m=1}^M\psi_\mu(m,\chi)
=&\sum_{m=1}^M\sum_{n<m}\chi(n)\Lambda(n)(m-n)^{\mu-1}
=\sum_{n=1}^M\chi(n)\Lambda(n)\sum_{m=n+1}^M(m-n)^{\mu-1}\\
=&\sum_{n=1}^M\chi(n)\Lambda(n)\sum_{m=1}^{M-n}m^{\mu-1}
=\sum_{n=1}^M\chi(n)\Lambda(n)\int_0^{M-n}u^{\mu-1}du+O(M).
\end{align*}
We can rewrite this as
\[=\sum_{n=1}^M\chi(n)\Lambda(n)
\int_n^M(M-u)^{\mu-1}du+O(M)
=\int_0^M\psi(u,\chi)(M-u)^{\mu-1}du+O(M).\]
Now we use the following explicit formula
which can be derived from Theorem 12.5 and Theorem 12.10 of \cite{MV}:
\[\psi(u,\chi)=E(\chi)u
-\sum_{\substack{\rho\\|\gamma|\le T}}\frac{u^\rho}{\rho}
+\frac{L'}{L}(1,\overline{\chi^\ast})
+O\left(\frac{M}{T}(\log qM)^2+(\log2q)(\log M)\right)\]
where $2\le u\le M,\ 2<T\le (qM)^2$, and
$\rho=\beta+i\gamma$ runs through the non-trivial zeros of $L(s,\chi)$
in $0<\sigma<1$, and the term $L'/L(1,\overline{\chi^\ast})$
appears only if $\chi\ppmod{q}$ is non-principal.
Note that we have removed the restriction that $\chi\ppmod{q}$ is primitive
as Theorem 12.5 and Theorem 12.10 of \cite{MV}
at the cost of (See (12.13) of \cite{MV})
\[(\log2q)(\log N).\]
By this explicit formula, we have
\begin{align*}
\sum_{m=1}^M\psi_\mu(m,\chi)
=&\int_0^M\psi(u,\chi)(M-u)^{\mu-1}du+O(M)\\
=&E(\chi)\int_0^Mu(M-u)^{\mu-1}du
-\sum_{\substack{\rho\\|\gamma|\le T}}\frac{1}{\rho}
\int_0^Mu^\rho(M-u)^{\mu-1}du\\
&\quad+\frac{M^\mu}{\mu}\cdot\frac{L'}{L}(1,\overline{\chi^\ast})
+O\left(\frac{M^{\mu+1}}{T}(\log qM)^2+M^\mu(\log2q)(\log M)\right)\\
=&E(\chi)\frac{M^{\mu+1}}{\mu(\mu+1)}
-\sum_{\substack{\rho\\|\gamma|\le T}}\frac{M^{\rho+\mu}}{\rho+\mu}
\cdot\frac{\Gamma(\rho)\Gamma(\mu)}{\Gamma(\rho+\mu)}\\
&\quad+\frac{M^\mu}{\mu}\cdot\frac{L'}{L}(1,\overline{\chi^\ast})
+O\left(\frac{M^{\mu+1}}{T}(\log qM)^2+M^\mu(\log2q)(\log M)\right).
\end{align*}
Now we extend the sum over non-trivial zeros
\[\sum_{\substack{\rho\\|\gamma|\le T}}\frac{M^{\rho+\mu}}{\rho+\mu}
\cdot\frac{\Gamma(\rho)\Gamma(\mu)}{\Gamma(\rho+\mu)}\]
to the sum over all non-trivial zeros.
In order to extend this sum, we recall the following estimate:
\[|\Gamma(s)|\asymp_A|t|^{\sigma-1/2}e^{-\pi|t|/2},
\quad|\sigma|\le A,\ |t|\ge 1\]
which can be derived from Stirling's formula.
Then we can estimate each term of the extended part of the sum by
\[\frac{M^{\rho+\mu}}{\rho+\mu}
\cdot\frac{\Gamma(\rho)\Gamma(\mu)}{\Gamma(\rho+\mu)}
\ll\frac{M^{\mu+1}}{|\gamma|^{\mu+1}}.\]
Hence we can extend the sum with the error
\[\ll M^{\mu+1}\sum_{\substack{\rho\\|\gamma|>T}}
\frac{1}{|\gamma|^{\mu+1}}
=M^{\mu+1}\int_T^\infty\frac{1}{t^{\mu+1}}dN(t,\chi)
\ll\frac{M^{\mu+1}}{T^\mu}(\log qT).\]
Summing up the above calculations, we have
\begin{align*}
\sum_{m=1}^M\psi_\mu(m,\chi)
=E(\chi)&\frac{M^{\mu+1}}{\mu(\mu+1)}
-\sum_\rho\frac{M^{\rho+\mu}}{\rho+\mu}
\cdot\frac{\Gamma(\rho)\Gamma(\mu)}{\Gamma(\rho+\mu)}
+\frac{M^\mu}{\mu}\cdot\frac{L'}{L}(1,\overline{\chi^\ast})\\
+&O\left(\frac{M^{\mu+1}}{T}(\log qM)^2+\frac{M^{\mu+1}}{T^\mu}(\log qT)
+M^\mu(\log2q)(\log M)\right).
\end{align*}
Taking $T=M(\log qM)$, we obtain
\begin{align*}
\sum_{m=1}^M\psi_\mu(m,\chi)
=E(\chi)\frac{M^{\mu+1}}{\mu(\mu+1)}
-\sum_\rho&\frac{M^{\rho+\mu}}{\rho+\mu}
\cdot\frac{\Gamma(\rho)\Gamma(\mu)}{\Gamma(\rho+\mu)}\\
&+\frac{M^\mu}{\mu}\cdot\frac{L'}{L}(1,\overline{\chi^\ast})
+O\left(M(\log2q)(\log M)\right)
\end{align*}
as we claimed.
\end{proof}
\section{Proof of the main theorem}
Now we prove the main theorem.
\begin{proof}[Proof of Theorem \ref{main_thm}]
Let us first define
\[R(n,\chi_1,\chi_2):=
\sum_{k_1+k_2=n}\chi_1(k_1)\Lambda(k_1)\chi_2(k_2)\Lambda(k_2)\]
for Dirichlet characters $\chi_1\pmod{q_1},\,\chi_2\pmod{q_2}$
and consider the mean value of this function
\[\sum_{n=1}^NR(n,\chi_1,\chi_2).\]
By the Fourier coefficient formula, we have
\[\sum_{m\le y}e^{-m/N}R(m,\chi_1,\chi_2)
=\int_{-1/2}^{1/2}\widetilde{S}(\alpha,\chi_1)\widetilde{S}(\alpha,\chi_2)
T(y,-\alpha)d\alpha\]
for $2\le y\le N$. Here introducing
\[\widetilde{R}(\alpha,\chi):=
\widetilde{S}(\alpha,\chi)-\frac{E(\chi)}{z}
+\sum_{\beta_\chi}\frac{\Gamma(\beta_\chi)}{z^{\beta_\chi}},\]
we can expand the above integral as
\begin{align*}
=&\sum_{i,j}\left(I_{E_iS_j}-\sum_{\beta_i}I_{\beta_iS_j}
+\sum_{\beta_j}I_{E_i\beta_j}\right)
-I_{E}-\sum_{\beta}I_{\beta_1\beta_2}+I_R,
\end{align*}
where in what follows, $(i,j)$ take values $(1,2)$ or $(2,1)$,
$\sum_{i,j}$ is the sum over such $(i,j)$'s,
$\beta_i$ runs through all real non-trivial zeros of $L(s,\chi_i)$
with $\beta_i\ge 1/2$ for $i=1,2$, and
\begin{gather*}
I_{E_iS_j}:=E(\chi_i)\int_{-1/2}^{1/2}T(y,-\alpha)
\frac{\widetilde{S}(\alpha,\chi_j)}{z}d\alpha,\quad
I_{\beta_iS_j}:=\Gamma(\beta_i)
\int_{-1/2}^{1/2}T(y,-\alpha)
\frac{\widetilde{S}(\alpha,\chi_j)}{z^{\beta_i}}d\alpha,\\
I_{E_i\beta_j}:=E(\chi_i)\Gamma(\beta_j)
\int_{-1/2}^{1/2}\frac{T(y,-\alpha)}{z^{\beta_j+1}}d\alpha,\quad
I_E:=E(\chi_1)E(\chi_2)
\int_{-1/2}^{1/2}\frac{T(y,-\alpha)}{z^2}d\alpha,\\
I_\beta:=\Gamma(\beta_1)\Gamma(\beta_2)
\int_{-1/2}^{1/2}\frac{T(y,-\alpha)}{z^{\beta_1+\beta_2}}d\alpha,\quad
I_R:=\int_{-1/2}^{1/2}T(y,-\alpha)
\widetilde{R}(\alpha,\chi_1)\widetilde{R}(\alpha,\chi_2)d\alpha.
\end{gather*}

First we calculate $I_{E_iS_j},\ I_{\beta_iS_j}$.
These can be calculated by Lemma \ref{detect} as
\[I_{E_iS_j}
=E(\chi_i)\sum_{m\le y}e^{-m/N}\psi(m-1,\chi_j)+O(N\log N),\]
\[I_{\beta_iS_j}=\sum_{m\le y}e^{-m/N}
\psi_{\beta_i}(m,\chi_j)+O(N\log N).\]
Next $I_{E_i\beta_j},\ I_E,\ I_\beta$
are calculated by Lemma \ref{t_detect} as
\[I_{E_i\beta_j}
=\frac{E(\chi_i)}{\beta_j}\sum_{m\le y}e^{-m/N}m^{\beta_j}+O(\log N),\]
\[I_{E_1E_2}
=E(\chi_1)E(\chi_2)\sum_{m\le y}e^{-m/N}m+O(\log N),\]
\[I_{\beta_1\beta_2}
=\frac{\Gamma(\beta_1)\Gamma(\beta_2)}{\Gamma(\beta_1+\beta_2)}
\sum_{m\le y}e^{-m/N}m^{\beta_1+\beta_2-1}+O(\log N).\]

We next estimate the error term $I_R$.
By the Caughy-Schwarz inequality, we have
\[I_R
\ll\left(\int_{-1/2}^{1/2}\left|T(y,-\alpha)\right|
\left|\widetilde{R}(\alpha,\chi_1)\right|^2d\alpha\right)^{1/2}
\left(\int_{-1/2}^{1/2}\left|T(y,-\alpha)\right|
\left|\widetilde{R}(\alpha,\chi_2)\right|^2d\alpha\right)^{1/2}.\]
Let us define
\[J_i:=\int_{-1/2}^{1/2}\left|T(y,-\alpha)\right|
\left|\widetilde{R}(\alpha,\chi_i)\right|^2d\alpha,\quad i=1,2.\]
Note that
\[T(y,-\alpha)\ll\min\left(y,\frac{1}{|\alpha|}\right)\]
for $|\alpha|\le 1/2$. Then Theorem \ref{LP_real} gives the estimate
\begin{align*}
J_i&\ll
y\int_{|\alpha|\le1/y}\left|\widetilde{R}(\alpha,\chi_i)\right|^2d\alpha
+\int_{1/y<|\alpha|\le1/2}
\frac{\left|\widetilde{R}(\alpha,\chi_i)\right|^2}{\alpha}d\alpha\\
&\ll N(\log q_iN)^2
+\sum_{k=0}^{O(\log y)}\frac{y}{2^k}
\int_{2^k/y<|\alpha|\le2^{k+1}/y}
\left|\widetilde{R}(\alpha,\chi_i)\right|^2d\alpha\\
&\ll N(\log q_iN)^2
+\sum_{k=0}^{O(\log y)}\frac{y}{2^k}\cdot\frac{2^{k+1}}{y}N(\log q_iN)^2
\ll N(\log N)(\log q_iN)^2
\end{align*}
for $i=1,2$. Hence we get
\[I_R\ll N(\log N)(\log q_1N)(\log q_2N).\]

Let us introduce
\begin{multline*}
E(m,\chi_1,\chi_2):=
\sum_{i,j}\left(E(\chi_i)\psi(m-1,\chi_j)
-\sum_{\beta_i}\psi_{\beta_i}(m,\chi_j)
+\sum_{\beta_j}\frac{E(\chi_i)}{\beta_j}m^{\beta_j}\right)\\
-E(\chi_1)E(\chi_2)m
-\sum_{\beta_1,\,\beta_2}
\frac{\Gamma(\beta_1)\Gamma(\beta_2)}{\Gamma(\beta_1+\beta_2)}
m^{\beta_1+\beta_2-1}
\end{multline*}
and
\[\Delta(m,\chi_1,\chi_2)=R(m,\chi_1,\chi_2)-E(m,\chi_1,\chi_2).\]
Then the above calculations give
\begin{equation}
\label{abelian_mean}
\sum_{m\le y}e^{-m/N}
\Delta(m,\chi_1,\chi_2)
\ll N(\log N)(\log q_1N)(\log q_2N)
\end{equation}
for  $y\le N$.
By partial summation, we have
\begin{align*}
\sum_{n\le N}\Delta(m,\chi_1,\chi_2)
=&\sum_{n\le N}e^{-n/N}\Delta(n,\chi_1,\chi_2)
\left(e-\frac{1}{N}\int_n^Ne^{y/N}dy\right)\\
=&e\sum_{n\le N}e^{-n/N}\Delta(n,\chi_1,\chi_2)\\
&\quad-\frac{1}{N}\int_1^Ne^{y/N}
\sum_{n\le y}e^{-n/N}\Delta(n,\chi_1,\chi_2)dy\\
\ll& N(\log N)(\log q_1N)(\log q_2N),
\end{align*}
in other words, we get
\begin{equation}
\label{pre_osc}
\sum_{n\le N}R(n,\chi_1,\chi_2)
=\sum_{n\le N}E(n,\chi_1,\chi_2)+O(N(\log N)(\log q_1N)(\log q_2N)).
\end{equation}

Next we shall calculate the right hand side of the last equation.
By Lemma \ref{cal_osc}, we get
\begin{align*}
E(\chi_i)\sum_{n=1}^N\psi(n-1,\chi_j)
=&E(\chi_i)\sum_{n=1}^N\psi(n,\chi_j)+O(N\log N)\\
=&E(\chi_i)E(\chi_j)\frac{N^2}{2}
-E(\chi_i)G(N,\chi_j)\\
&\quad+E(\chi_i)N\cdot\frac{L'}{L}(1,\overline{\chi_j^\ast})
+O(N(\log N)(\log2q_j)),\\
\sum_{n=1}^N\sum_{\beta_i}\psi_{\beta_i}(n,\chi_j)
=E(\chi_j)&\sum_{\beta_i}\frac{N^{\beta_i+1}}{\beta_i(\beta_i+1)}
-\sum_{\beta_i}G^{\beta_i}(N,\chi_j)\\
&+\sum_{\beta_i}\frac{N^{\beta_i}}{\beta_i}
\cdot\frac{L'}{L}(1,\overline{\chi_j^\ast})
+O(N(\log N)(\log2q_i)(\log2q_j)).
\end{align*}
Now let us recall that for $0<d$,
\[\sum_{n\le x}n^{d-1}=\frac{1}{d}x^d+O(1+x^{d-1})\]
holds. So we have
\[\sum_{n=1}^N\sum_{\beta_j}\frac{E(\chi_i)}{\beta_j}n^{\beta_j}
=\sum_{\beta_j}E(\chi_i)\frac{N^{\beta_j+1}}{\beta_j(\beta_j+1)}
+O(N(\log2q_j)),\]
\[\sum_{n=1}^N\sum_{\beta_1,\,\beta_2}
\frac{\Gamma(\beta_1)\Gamma(\beta_2)}{\Gamma(\beta_1+\beta_2)}
n^{\beta_1+\beta_1-1}
=\sum_{\beta_1,\,\beta_2}W(N,\beta_1,\beta_2)+O(N(\log2q_1)(\log2q_2)),\]
and
\[E(\chi_1)E(\chi_2)\sum_{n=1}^Nn=E(\chi_1)E(\chi_2)\frac{N^2}{2}+O(N).\]

Summing up these formulae, we get
\begin{align}
\sum_{n\le N}R(n,\chi_1,\chi_2)
=&E(\chi_1)E(\chi_2)\frac{N^2}{2}
-E(\chi_2)G(N,\chi_1)-E(\chi_1)G(N,\chi_2)\notag\\
+&\sum_{\beta_2}G^{\beta_2}(N,\chi_1)
+\sum_{\beta_1}G^{\beta_1}(N,\chi_2)
-\sum_{\beta_1,\,\beta_2}W(N,\beta_1,\beta_2)\notag\\
\label{character_main}
&\quad+E(\chi_2)N\cdot\frac{L'}{L}(1,\overline{\chi_1^\ast})
+E(\chi_1)N\cdot\frac{L'}{L}(1,\overline{\chi_2^\ast})\\
&\quad-\sum_{\beta_2}\frac{N^{\beta_2}}{\beta_2}
\cdot\frac{L'}{L}(1,\overline{\chi_1^\ast})
-\sum_{\beta_1}\frac{N^{\beta_1}}{\beta_1}
\cdot\frac{L'}{L}(1,\overline{\chi_2^\ast})\notag\\
&\quad\quad\quad\quad\quad\quad\quad\quad
+O(N(\log N)(\log q_1N)(\log q_2N)).\notag
\end{align}

Multiplying (\ref{character_main}) by
\[\frac{\overline{\chi_1}(a_1)\overline{\chi_2}(a_2)}{\phi(q_1)\phi(q_2)},\]
and summing up over characters $\chi_1\ppmod{q_1},\ \chi_2\ppmod{q_2}$,
we have
\begin{multline*}
\sum_{n\le N}R(n,q_1,a_1,q_2,a_2)
=\frac{1}{\phi(q_1)\phi(q_2)}
\left(\frac{N^2}{2}-G(N,q_1,a_2)-G(N,q_2,a_2)+H(N)+S(N)\right)\\
+O(N(\log N)(\log q_1N)(\log q_2N)),
\end{multline*}
where $S(N)$ is defined by
\begin{equation}
\begin{aligned}
\label{S_def}
S(N):=&
N\sum_{\chi_1\neq\chi_0\ppmod{q_1}}\frac{L'}{L}(1,\overline{\chi_1^\ast})
+N\sum_{\chi_2\neq\chi_0\ppmod{q_2}}\frac{L'}{L}(1,\overline{\chi_2^\ast})\\
&-\sum_{\substack{\chi_1\neq\chi_0\ppmod{q_1}\\\chi_2\neq\chi_0\ppmod{q_2}}}
\left(\frac{L'}{L}(1,\overline{\chi_1^\ast})\cdot\overline{\chi_2}(a_2)
\sum_{\beta_2}\frac{N^{\beta_2}}{\beta_2}
+\frac{L'}{L}(1,\overline{\chi_2^\ast})\cdot\overline{\chi_1}(a_1)
\sum_{\beta_1}\frac{N^{\beta_1}}{\beta_1}\right).
\end{aligned}
\end{equation}
These terms are affected by Siegel zeros but not much bigger than the error term.
This can be estimated as
\begin{equation}
\label{final_estimate}
\frac{S(N)}{\phi(q_1)\phi(q_2)}\ll N(\log2q_1)(\log2q_2)
\end{equation}
by recalling Theorem \ref{Siegel}, Lemma \ref{log_dev}, the well-known estimate
\[\phi(q)\gg\frac{q}{\log\log4q},\]
and the famous fact \cite[p.370, Corollary 11.12]{MV}
\[\frac{1}{1-\beta_0(\chi_i)}\ll q_i^{1/2}(\log q_i)^2,\ i=1,2.\]
Substituting (\ref{final_estimate}) into the above formula, we finally arrive at
\begin{multline*}
\sum_{n\le N}R(n,q_1,a_1,q_2,a_2)
=\frac{1}{\phi(q_1)\phi(q_2)}
\left(\frac{N^2}{2}-G(N,q_1,a_2)-G(N,q_2,a_2)+H(N)\right)\\
+O(N(\log N)(\log q_1N)(\log q_2N)).
\end{multline*}
The restriction that our argument $N$ is integer can be removed
by considering the variation of each term
while the argument varies over bounded intervals.
\end{proof}

\begin{xrem}
Our Theorem \ref{main_thm} conflicts with R\"uppel's continuation.
The inconsistency happens on the residues of the poles
\[\rho_\psi+\beta_1,\quad \rho_\chi+\beta_2\]
of the function $\Phi_2(s,a,b,q)$.
The difference between our result and hers is only up to the sign and
this happened because of her minor mistakes on the determination of
the sign of residues
\[-\frac{L'}{L}(s-\rho_\psi,\chi)\]
at \cite[p.30]{Ruppel1} or \cite[p.142]{Ruppel2}.
\end{xrem}

\subsection*{Acknowledgements}
The author would like to thank Professor Kohji Matsumoto
for his suggestion of this problem, his advices and his encouragement.


\begin{thebibliography}{HD}

\bibitem{BP1} G. Bhowmik, J-C. Schlage-Puchta,
\emph{Mean representation number of integers as the sum of primes},
Nagoya Math. Journal \textbf{200} (2010), 27--33.
\bibitem{EM} S. Egami, K. Matsumoto,
\emph{Convolutions of the von Mangoldt function and related Dirichlet series},
Proceedings of the 4th China-Japan Seminar held at Shangdong, 1--23,
S. Kanemitsu and J. -Y. Liu eds., World Sci. Publ., Hackensack, NJ, (2007).
\bibitem{Fujii1} A. Fujii, \emph{An additive problem of prime numbers},
Acta Arith. \textbf{58} (1991), 173--179.
\bibitem{Fujii2} A. Fujii, \emph{An additive problem of prime numbers},
Proc. Japan Acad. Ser. A Math. Sci. \textbf{67} (1991), 248--252.
\bibitem{LP} A. Languasco and A. Perelli,
\emph{On Linnik's theorem on Goldbach numbers in short intervals
and related problems},
Ann. Inst. Fourier. \textbf{44} (1994), 307--322.
\bibitem{LZG} A. Languasco and A. Zaccagnini,
\emph{The number of Goldbach representations of an integer},
Proc. Amer. Math. Soc. \textbf{140} (2012), 795--804,
\url{http://arxiv.org/abs/1011.3198}.
\bibitem{LZ_many} A. Languasco and A. Zaccagnini,
\emph{Sums of many primes},
Journal of Number Theory \textbf{132} (2012), 1265--1283.
\bibitem{Lavrik} A. F. Lavrik,
\emph{The number of $k$-twin primes lying on an interval of given length},
Dokl. Akad. Nauk SSSR \textbf{136} (1961), 281--283 (in Russian);
English transl.: Soviet Math. Dokl. 2 (1961), 52--55.
\bibitem{LIU_ZHAN} M. C. Liu and T. Zhan,
\emph{The Goldbach problem with primes in arithmetic progressions}, in:
Analaytic Number Theory (Kyoto, 1996), London Math. Soc. Lecture Note Ser. 247,
Cambridge Univ. Press, (1997), 227--251.
\bibitem{MV} H. L. Montgomery and R. C. Vaughan,
\emph{Multiplicative Number Theory I. Classical Theory},
Cambridge University Press, (2007).
\bibitem{Ruppel1} F. R\"uppel,
\emph{Convolution of the von Mangoldt function over residue classes},
Diplomarbeit, Univ. W\"urzburg, (2009).
\bibitem{Ruppel2} F. R\"uppel,
\emph{Convolution of the von Mangoldt function over residue classes},
\v{S}iauliai Math. Semin. \textbf{7 (15)} (2012), 135--156.
\end{thebibliography}
\end{document}